\documentclass[conference]{IEEEtran}
\usepackage{amsmath,amsfonts,amssymb,amsthm,txfonts,hyperref,tikz,balance}
\newtheorem{lemma}{Lemma}
\newtheorem{corollary}{Corollary}
\newtheorem{theorem}{Theorem}
\newcommand{\N}{\mathbb N}
\newcommand{\Z}{\mathbb Z}
\newcommand{\K}{\textbf{\textit{K}}}
\title{A hypothetical way to compute an upper bound for the heights
of solutions of a Diophantine equation with a finite number of solutions}
\author{
\IEEEauthorblockN{Apoloniusz Tyszka}
\IEEEauthorblockA{
University of Agriculture\\
Faculty of Production and Power Engineering\\
Balicka 116B, 30-149 Krak\'ow, Poland\\
Email: rttyszka@cyf-kr.edu.pl}}
\begin{document}
\maketitle
\begin{abstract}
Let
\begin{displaymath}
f(n)=\left\{
\begin{array}{ccl}
1 & {\rm if} & n=1 \\
2^{\textstyle 2^{n-2}} & {\rm if} & n \in \{2,3,4,5\} \\
\left(2+2^{\textstyle 2^{n-4}}\right)^{\textstyle 2^{n-4}} & {\rm if} & n \in \{6,7,8,\ldots\}
\end{array}
\right.
\end{displaymath}
\vskip 0.01truecm
\noindent
We conjecture that if a system
\[
T \subseteq \{x_i+1=x_k,~x_i \cdot x_j=x_k:~i,j,k \in \{1,\ldots,n\}\}
\]
has only finitely many solutions in positive integers \mbox{$x_1,\ldots,x_n$},
then each such solution \mbox{$(x_1,\ldots,x_n)$} satisfies \mbox{$x_1,\ldots,x_n \leqslant f(n)$}.
We prove that the function $f$ cannot be decreased and the conjecture implies that
there is an algorithm which takes as input a Diophantine equation, returns an integer,
and this integer is greater than the heights of integer (\mbox{non-negative} integer,
positive integer, rational) solutions, if the solution set is finite.
We show that if the conjecture is true, then this can be partially confirmed by the
execution of a \mbox{brute-force} algorithm.
\end{abstract}
\par
\IEEEoverridecommandlockouts
\IEEEPARstart{I}{n this} article, we present a conjecture on integer arithmetic which
implies a positive answer to all versions of the following open problem:
\vskip 0.2truecm
\noindent
{\bf Problem.} {\em Is there an algorithm which takes as input a Diophantine equation, returns an integer,
and this integer is greater than the heights of integer (\mbox{non-negative} integer, positive integer, rational)
solutions, if the solution set is finite?}
\vskip 0.2truecm
\par
We remind the reader that the height of a rational number $\frac{p}{q}$ is defined by \mbox{${\rm max}(|p|,|q|)$}
provided $\frac{p}{q}$ is written in lowest terms.
\begin{theorem}\label{the1}
Only \mbox{$x_1=1$} solves the equation \mbox{$x_1 \cdot x_1=x_1$} in positive integers.
Only \mbox{$x_1=1$} and \mbox{$x_2=2$} solve the system \mbox{$\{x_1 \cdot x_1=x_1,~x_1+1=x_2\}$}
in positive integers. For each integer \mbox{$n \geqslant 3$}, the following system
\begin{displaymath}
\left\{
\begin{array}{rcl}
x_1 \cdot x_1 &=& x_1 \\
x_1+1 &=& x_2 \\
\forall i \in \{2,\ldots,n-1\} ~x_i \cdot x_i &=& x_{i+1}
\end{array}
\right.
\end{displaymath}
\noindent
has a unique solution in positive integers, namely\\
$\left(1,2,4,16,256,\ldots,2^{\textstyle 2^{n-3}},2^{\textstyle 2^{n-2}}\right)$.
\end{theorem}
\begin{theorem}\label{the2}
For each positive integer $n$, the following system
\[
\left\{\begin{array}{rcl}
\forall i \in \{1,\ldots,n\} ~x_i \cdot x_i &=& x_{i+1} \\
x_{n+2}+1 &=& x_1 \\
x_{n+3}+1 &=& x_{n+2} \\
x_{n+3} \cdot x_{n+4} &=& x_{n+1}
\end{array}\right.
\]
is soluble in positive integers and has only finitely many integer solutions.
Each integer solution \mbox{$(x_1,\ldots,x_{n+4})$} satisfies
\mbox{$|x_1|,\ldots,|x_{n+4}| \leqslant \left(2+2^{\textstyle 2^n}\right)^{\textstyle 2^n}$}.
The maximal solution in positive integers is given by
\begin{displaymath}
\left\{\begin{array}{rcl}
\forall i \in \{1,\ldots,n+1\} ~x_i &=& \left(2+2^{\textstyle 2^n}\right)^{\textstyle 2^{i-1}} \\
x_{n+2} &=& 1+2^{\textstyle 2^n} \\
x_{n+3} &=& 2^{\textstyle 2^n} \\
x_{n+4} &=& \left(1+2^{\textstyle 2^n-1}\right)^{\textstyle 2^n}
\end{array}\right.
\end{displaymath}
\end{theorem}
\begin{proof}
The system equivalently expresses that
$(x_1-2) \cdot x_{n+4}=x_1^{\textstyle 2^n}$.
By this and the polynomial identity
\[
x_1^{\textstyle 2^n}=2^{\textstyle 2^n}+(x_1-2) \cdot \sum_{\textstyle k=0}^{\textstyle 2^n-1} 2^{\textstyle 2^n-1-k} \cdot x_1^k
\]
we get that \mbox{$x_{n+3}=x_1-2$} divides $2^{\textstyle 2^n}$ and
\mbox{$x_{n+4}=\frac{\textstyle x_1^{\textstyle 2^n}}{\textstyle x_1-2}$}. Hence,
\mbox{$x_1 \in \left[2-2^{\textstyle 2^n},2+2^{\textstyle2^n}\right] \cap \Z$},
the system has only finitely many integer solutions, and
$|x_1|,\ldots,|x_{n+4}| \leqslant \left(2+2^{\textstyle 2^n}\right)^{\textstyle 2^n}$.
\end{proof}
\par
In~\cite[p.~719]{Tyszka2}, the author proposed the upper bound $2^{\textstyle 2^{n-1}}$
for positive integer solutions to any system
\[
T \subseteq \{x_i+x_j=x_k,~x_i \cdot x_j=x_k:~i,j,k \in \{1,\ldots,n\}\}
\]
which has only finitely many solutions in positive integers \mbox{$x_1,\ldots,x_n$}.
The bound $2^{\textstyle 2^{n-1}}$ is not correct for any \mbox{$n \geqslant 8$} because
the following system
\[
\left\{\begin{array}{rcl}
\forall i \in \{1,\ldots,k\} ~x_i \cdot x_i &=& x_{i+1} \\
x_{k+2}+x_{k+2} &=& x_{k+3} \\
x_{k+2} \cdot x_{k+2} &=& x_{k+3} \\
x_{k+4}+x_{k+3} &=& x_1 \\
x_{k+4} \cdot x_{k+5} &=& x_{k+1}
\end{array}\right.
\]
provides a counterexample for any \mbox{$k \geqslant 3$}. In \mbox{\cite[p.~96]{Tyszka3}},
the author proposed the upper bound $2^{\textstyle 2^{n-1}}$ for modulus of integer
solutions to any system
\[
T \subseteq \{x_k=1,~x_i+x_j=x_k,~x_i \cdot x_j=x_k:~i,j,k \in \{1,\ldots,n\}\}
\]
which has only finitely many solutions in integers \mbox{$x_1,\ldots,x_n$}.
The bound $2^{\textstyle 2^{n-1}}$ is not correct for any \mbox{$n \geqslant 9$} because
the following system
\[
\left\{\begin{array}{rcl}
\forall i \in \{1,\ldots,k\} ~x_i \cdot x_i &=& x_{i+1} \\
x_{k+2} &=& 1 \\
x_{k+3}+x_{k+2} &=& x_1 \\
x_{k+4}+x_{k+2} &=& x_{k+3} \\
x_{k+4} \cdot x_{k+5} &=& x_{k+1}
\end{array}\right.
\]
provides a counterexample for any \mbox{$k \geqslant 4$}. Let
\begin{displaymath}
f(n)=\left\{
\begin{array}{ccl}
1 & {\rm if} & n=1 \\
2^{\textstyle 2^{n-2}} & {\rm if} & n \in \{2,3,4,5\} \\
\left(2+2^{\textstyle 2^{n-4}}\right)^{\textstyle 2^{n-4}} & {\rm if} & n \in \{6,7,8,\ldots\}
\end{array}
\right.
\end{displaymath}
\noindent
{\bf Conjecture.} {\em If a system
\vskip 0.1truecm
\noindent
\centerline{$T \subseteq \{x_i+1=x_k,~x_i \cdot x_j=x_k:~i,j,k \in \{1,\ldots,n\}\}$}
\vskip 0.1truecm
\noindent
has only finitely many solutions in positive integers \mbox{$x_1,\ldots,x_n$}, then each such solution
\mbox{$(x_1,\ldots,x_n)$} satisfies \mbox{$x_1,\ldots,x_n \leqslant f(n)$}.}
\vskip 0.2truecm
\par
Theorems~\ref{the1} and~\ref{the2} imply that the function $f$ cannot be decreased.
Let \mbox{$\cal{R}${\sl ng}} denote the class of all rings $\K$ that extend~$\Z$, and let
\[
E_n=\{x_k=1,~x_i+x_j=x_k,~x_i \cdot x_j=x_k:~i,j,k \in \{1,\ldots,n\}\}
\]
\mbox{Th. Skolem} proved that any Diophantine equation can be algorithmically transformed
into an equivalent system of Diophantine equations of degree at most~$2$,
see \mbox{\cite[pp.~2--3]{Skolem}} and \mbox{\cite[pp.~3--4]{Matiyasevich1}}.
The following result strengthens Skolem's theorem.
\begin{lemma}\label{lem1} (\cite[p.~720]{Tyszka2})
Let \mbox{$D(x_1,\ldots,x_p) \in {\Z}[x_1,\ldots,x_p]$}.
Assume that \mbox{${\rm deg}(D,x_i) \geqslant 1$} for each \mbox{$i \in \{1,\ldots,p\}$}. We can compute a positive
integer \mbox{$n>p$} and a system \mbox{$T \subseteq E_n$} which satisfies the following two conditions:
\vskip 0.2truecm
\noindent
{\tt Condition 1.} If \mbox{$\K \in {\cal R}{\sl ng} \cup \{\N,~\N \setminus \{0\}\}$}, then
\[
\forall \tilde{x}_1,\ldots,\tilde{x}_p \in \K ~\Bigl(D(\tilde{x}_1,\ldots,\tilde{x}_p)=0 \Longleftrightarrow
\]
\[
\exists \tilde{x}_{p+1},\ldots,\tilde{x}_n \in \K ~(\tilde{x}_1,\ldots,\tilde{x}_p,\tilde{x}_{p+1},\ldots,\tilde{x}_n) ~solves~ T\Bigr)
\]
{\tt Condition 2.} If \mbox{$\K \in {\cal R}{\sl ng} \cup \{\N,~\N \setminus \{0\}\}$}, then
for each \mbox{$\tilde{x}_1,\ldots,\tilde{x}_p \in \K$} with \mbox{$D(\tilde{x}_1,\ldots,\tilde{x}_p)=0$},
there exists a unique tuple \mbox{$(\tilde{x}_{p+1},\ldots,\tilde{x}_n) \in {\K}^{n-p}$} such that the tuple
\mbox{$(\tilde{x}_1,\ldots,\tilde{x}_p,\tilde{x}_{p+1},\ldots,\tilde{x}_n)$} solves $T$.
\vskip 0.2truecm
\noindent
Conditions 1 and 2 imply that for each \mbox{$\K \in {\cal R}{\sl ng} \cup \{\N,~\N \setminus \{0\}\}$},
the equation \mbox{$D(x_1,\ldots,x_p)=0$} and the system $T$ have the same number of solutions in $\K$.
\end{lemma}
\par
For a positive integer $n$, let $S(n)$ denote the successor of $n$.
\begin{lemma}\label{lem2}
Let $T$ be a finite system of equations of the forms: \mbox{$x=1$}, \mbox{$x+y=z$},
and \mbox{$x \cdot y=z$}. If the equation \mbox{$x=1$} belongs to $T$, then the system
\mbox{$T \cup \{x \cdot x=x\} \setminus \{x=1\}$} has the same solutions in positive
integers.
\end{lemma}
\begin{lemma}\label{lem3}
Let $T$ be a finite system of equations of the forms: \mbox{$S(x)=y$}, \mbox{$x+y=z$},
and \mbox{$x \cdot y=z$}. If the equation \mbox{$x+y=z$} belongs to $T$ and
the variables $z_1$, $z_2$, $\widetilde{z_1}$, $\widetilde{z_2}$, $\widetilde{v}$,
$u$, $t$, $\widetilde{t}$, $v$ are new, then the following system
\[
T \cup \{z \cdot x=z_1,~z \cdot y=z_2,~S(z_1)=\widetilde{z_1},~S(z_2)=\widetilde{z_2},
~\widetilde{z_1} \cdot \widetilde{z_2}=\widetilde{v},
\]
\[
z \cdot z=u,~x \cdot y=t,~S(t)=\widetilde{t},~u \cdot \widetilde{t}=v,~S(v)=\widetilde{v}\} \setminus \{x+y=z\}
\]
has the same solutions in positive integers and a smaller number of additions.
\end{lemma}
\begin{proof}
According to \cite[p.~100]{Robinson}, for each positive integers \mbox{$x,y,z$}, \mbox{$x+y=z$}
if and only if
\[
S(z \cdot x) \cdot S(z \cdot y)=S((z \cdot z) \cdot S(x \cdot y))
\]
Indeed, the above equality is equivalent to
\[
\left(z^2 \cdot x \cdot y+1\right)+z \cdot (x+y)=\left(z^2 \cdot x \cdot y+1\right)+z^2
\]
\end{proof}
\par
Lemmas~\ref{lem1}--\ref{lem3} imply the next theorem.
\begin{theorem}\label{the3}
If we assume the Conjecture and a Diophantine equation \mbox{$D(x_1,\ldots,x_p)=0$} has only
finitely many solutions in positive integers, then an upper bound for these solutions
can be computed.
\end{theorem}
\begin{corollary}\label{cor1}
If we assume the Conjecture and a Diophantine equation \mbox{$D(x_1,\ldots,x_p)=0$}
has only finitely many solutions in \mbox{non-negative} integers, then an upper bound for
these solutions can be computed by applying Theorem~\ref{the3} to the equation
\mbox{$D(x_{1}-1,\ldots,x_{p}-1)=0$}.
\end{corollary}
\begin{corollary}\label{cor2}
If we assume the Conjecture and a Diophantine equation \mbox{$D(x_1,\ldots,x_p)=0$} has only finitely many
integer solutions, then an upper bound for their modulus can be computed by applying Theorem~\ref{the3}
to the equation
\[
\prod\limits_{\textstyle (i_1,\ldots,i_p) \in \{1,2\}^p} D((-1)^{i_1} \cdot (x_{1}-1),\ldots,(-1)^{i_p} \cdot (x_{p}-1))=0
\]
\end{corollary}
\vskip 0.2truecm
\par
\begin{lemma}\label{lem4} (\cite[p.~720]{Tyszka2})
If there is a computable upper bound for the modulus of integer solutions to a Diophantine
equation with a finite number of integer solutions, then there is a computable upper bound
for the heights of rational solutions to a Diophantine equation with a finite number of
rational solutions.
\end{lemma}
\begin{theorem}\label{the4}
The Conjecture implies that there is a computable upper bound for the heights of rational
solutions to a Diophantine equation with a finite number of rational solutions.
\end{theorem}
\begin{proof}
It follows from Corollary~\ref{cor2} and Lemma~\ref{lem4}.
\end{proof}
\par
The Davis-Putnam-Robinson-Matiyasevich theorem states that every recursively
enumerable set \mbox{${\cal M} \subseteq {\N}^n$} has a Diophantine
representation, that is
\[
(a_1,\ldots,a_n) \in {\cal M} \Longleftrightarrow
\]
\[
\exists x_1, \ldots, x_m \in \N ~~W(a_1,\ldots,a_n,x_1,\ldots,x_m)=0 \tag*{\texttt{(R)}}
\]
for some polynomial $W$ with integer coefficients, see \cite{Matiyasevich1}.
The polynomial~$W$ can be computed, if we know the Turing \mbox{machine $M$} such
that, for all \mbox{$(a_1,\ldots,a_n) \in {\N}^n$}, $M$ halts on \mbox{$(a_1,\ldots,a_n)$} if
and only if \mbox{$(a_1,\ldots,a_n) \in {\cal M}$}, \mbox{see \cite{Matiyasevich1}}.
The representation~\texttt{(R)} is said to be \mbox{finite-fold}, if for any
\mbox{$a_1,\ldots,a_n \in \N$} the equation
\mbox{$W(a_1,\ldots,a_n,x_1,\ldots,x_m)=0$} has only finitely many solutions
\mbox{$(x_1,\ldots,x_m) \in {\N}^m$}. \mbox{Yu. Matiyasevich} conjectures that
each recursively enumerable set \mbox{${\cal M} \subseteq {\N}^n$} has a
\mbox{finite-fold} Diophantine representation, see \mbox{\cite[pp.~341--342]{DMR}},
\mbox{\cite[p.~42]{Matiyasevich2}}, and \mbox{\cite[p.~745]{Matiyasevich3}}.
Matiyasevich's conjecture implies a negative answer to the Problem, see \mbox{\cite[p.~42]{Matiyasevich2}}.
\begin{theorem}\label{the5} (cf.~\cite[p.~721]{Tyszka2})
The Conjecture implies that if a set \mbox{${\cal M} \subseteq \N$} has a \mbox{finite-fold}
Diophantine representation, then ${\cal M}$ is computable.
\end{theorem}
\begin{proof}
Let a set \mbox{${\cal M} \subseteq \N$} has a \mbox{finite-fold} Diophantine representation.
It means that there exists a polynomial \mbox{$W(x,x_1,\ldots,x_m)$} with integer coefficients such that
\[
\forall b \in \N~ \Bigl(b \in {\cal M} \Longleftrightarrow \exists x_1, \ldots, x_m \in \N ~~W(b,x_1,\ldots,x_m)=0\Bigr)
\]
and for any \mbox{$b \in \N$} the equation \mbox{$W(b,x_1,\ldots,x_m)=0$} has only finitely many solutions
\mbox{$(x_1,\ldots,x_m) \in {\N}^m$}. By Corollary~\ref{cor1}, there is a computable function \mbox{$g \colon \N \to \N$}
such that for each \mbox{$b,x_1,\ldots,x_m \in \N$} the equality \mbox{$W(b,x_1,\ldots,x_m)=0$} implies
\mbox{${\rm max}(x_1,\ldots,x_m) \leqslant g(b)$}.
Hence, we can decide whether or not a \mbox{non-negative} integer $b$ belongs to ${\cal M}$ by checking
whether or not the equation \mbox{$W(b,x_1,\ldots,x_m)=0$} has an integer solution in the box \mbox{$[0,g(b)]^m$}.
\end{proof}
\par
In this paragraph, we follow \cite{Tyszka1} and
we explain why Matiyasevich's conjecture although widely known is less widely accepted.
Let us say that a set \mbox{${\cal M} \subseteq {\N}^n$} has a bounded Diophantine
representation, if there exists a \mbox{polynomial $W$} with integer coefficients such that
\[
(a_1,\ldots,a_n) \in {\cal M} \Longleftrightarrow \exists x_1,\ldots,x_m \in \left\{0,\ldots,\mathrm{max}\left(a_1,\ldots,a_n\right)\right\}
\]
\[
W\left(a_1,\ldots,a_n,x_1,\ldots,x_m\right)=0
\]
Of course, any bounded Diophantine representation is \mbox{finite-fold}
and any subset of $\N$ with a bounded Diophantine
representation is computable. A simple diagonal argument shows that there exists
a computable subset of $\N$ without any bounded Diophantine
representation, see \cite[p.~360]{DMR}. The authors of \cite{DMR}
suggest a possibility (which contradicts Matiyasevich's conjecture)
that each subset of $\N$ which has a \mbox{finite-fold} Diophantine representation
has also a bounded Diophantine representation, see \cite[p.~360]{DMR}.
\vskip 0.2truecm
\par
For a positive integer $n$, let $\tau(n)$ denote the smallest positive
integer $b$ such that for each system
\[
T \subseteq \{x_i+1=x_k,~x_i \cdot x_j=x_k:~i,j,k \in \{1,\ldots,n\}\}
\]
with a finite number of solutions in positive integers \mbox{$x_1,\ldots,x_n$},
all these solutions belong to \mbox{$[1,b]^n$}. By Theorems~\ref{the1} and~\ref{the2},
\mbox{$f(n) \leqslant \tau(n)$} for any positive integer $n$. The Conjecture
implies that \mbox{$f=\tau$}.
\begin{theorem}\label{the6} (cf. \cite[Theorem 4]{Tyszka1})
If a function \mbox{$h\colon \N \setminus \{0\} \to \N \setminus \{0\}$}
has a \mbox{finite-fold} Diophantine representation,
then there exists a positive integer $m$ such that \mbox{$h(n)<\tau(n)$} for any \mbox{$n>m$}.
\end{theorem}
\begin{proof}
There exists a polynomial \mbox{$W(x_1,x_2,x_3,\ldots,x_r)$} with integer coefficients
such that for each positive integers \mbox{$x_1,x_2$},
\[
(x_1,x_2) \in h \Longleftrightarrow
\]
\[
\exists x_3,\ldots,x_r \in \N \setminus \{0\}~~ W(x_1,x_2,x_3-1,\ldots,x_r-1)=0
\]
and for each positive integers \mbox{$x_1,x_2$} at most finitely many tuples
\mbox{$(x_3,\ldots,x_r)$} of positive integers satisfy \mbox{$W(x_1,x_2,x_3-1,\ldots,x_r-1)=0$}.
By \mbox{Lemmas \ref{lem1}--\ref{lem3}}, there is an integer \mbox{$s \geqslant 3$} such that
for any positive integers \mbox{$x_1,x_2$},
\[
(x_1,x_2) \in h \Longleftrightarrow
\]
\begin{equation}
\exists x_3,\ldots,x_s \in \N \setminus \{0\}~~ \Psi(x_1,x_2,x_3,\ldots,x_s)\tag*{\texttt{(E)}}
\end{equation}
where \mbox{$\Psi(x_1,x_2,x_3,\ldots,x_s)$} is a conjunction of formulae of the forms
\mbox{$x_i+1=x_k$} and \mbox{$x_i \cdot x_j=x_k$}, the indices $i,j,k$ belong to
$\{1,\ldots,s\}$, and for each positive integers \mbox{$x_1,x_2$} at most finitely
many tuples \mbox{$(x_3,\ldots,x_s)$} of positive integers satisfy \mbox{$\Psi(x_1,x_2,x_3,\ldots,x_s)$}.
Let $[\cdot]$ denote the integer part function, and let an integer $n$ is greater than \mbox{$m=2s+2$}.
Then,
\[
n \geqslant \left[\frac{n}{2}\right]+\frac{n}{2}>\left[\frac{n}{2}\right]+s+1
\]
and \mbox{$n-\left[\frac{n}{2}\right]-s-2 \geqslant 0$}.
Let $T_n$ denote the following system with $n$ variables:
\[
\left\{
\begin{array}{c}
\textrm{all~equations~occurring~in~}\Psi(x_1,x_2,x_3,\ldots,x_s)\\
\begin{array}{rcl}
\forall i \in \left\{1,\ldots,n-\left[\frac{n}{2}\right]-s-2\right\} ~u_i \cdot u_i &=& u_i\\
t_1 \cdot t_1 &=& t_1\\
\forall i \in \left\{1,\ldots,\left[\frac{n}{2}\right]-1\right\} ~t_i+1 &=& t_{i+1}\\
t_2 \cdot t_{\left[\frac{n}{2}\right]} &=& u\\
u+1 &=&x_1 {\rm ~(if~}n{\rm ~is~odd)}\\
t_1 \cdot u &=& x_1 {\rm ~(if~}n{\rm ~is~even)}\\
x_2+1 &=& y
\end{array}
\end{array}
\right.
\]
By the equivalence~\texttt{(E)}, the system $T_n$ is soluble in positive integers,
\mbox{$2 \cdot \left[\frac{n}{2}\right]=u$}, \mbox{$n=x_1$}, and
\[
h(n)=h(x_1)=x_2<x_2+1=y
\]
Since $T_n$ has at most finitely many solutions in positive integers, \mbox{$y \leqslant \tau(n)$}.
Hence, \mbox{$h(n)<\tau(n)$}.
\end{proof}
\newpage
Below is the excerpt from page 135 of the \mbox{book \cite{Smart}:}
\vskip 0.2truecm
\noindent
{\em Folklore. If a Diophantine equation has only finitely many solutions then those solutions
are small in `height' when compared to the parameters of the equation.
\vskip 0.2truecm
\par
This folklore is, however, only widely believed because of the large amount of experimental
evidence which now exists to support it.}
\vskip 0.2truecm
\par
Below is the excerpt from page 12 of the article \cite{Stoll}:
\vskip 0.2truecm
\noindent
{\em Note that if a Diophantine equation is solvable, then we can prove it, since we will
eventually find a solution by searching through the countably many possibilities
(but we do not know beforehand how far we have to search). So the really hard
problem is to prove that there are no solutions when this is the case. A similar
problem arises when there are finitely many solutions and we want to find them
all. In this situation one expects the solutions to be fairly small. So usually it
is not so hard to find all solutions; what is difficult is to show that there are no
others.}
\vskip 0.2truecm
\par
That is, mathematicians are intuitively persuaded that solutions are small when there are finitely
many of them. It seems that there is a reason which is common to all the equations. Such a
reason might be the Conjecture whose consequences we have already presented.
\vskip 0.2truecm
\par
For a positive integer $b$, let $\Phi(b)$ denote the Conjecture restricted to systems whose
all solutions in positive integers are not greater than $b$. Obviously,
\[
\Phi(1) \Leftarrow \Phi(2) \Leftarrow \Phi(3) \Leftarrow \ldots
\]
and the Conjecture is equivalent to \mbox{$\forall b \in \N \setminus \{0\}~\Phi(b)$}.
For each positive integer $n$, there are only finitely many systems
\[
T \subseteq \{x_i+1=x_k,~x_i \cdot x_j=x_k:~i,j,k \in \{1,\ldots,n\}\}
\]
Hence, for each positive integer $n$ there exists a positive integer $m$ such that
the Conjecture restricted to systems with at most $n$ variables is equivalent to
the sentence $\Phi(m)$. The Conjecture is true for \mbox{$n=1$} and \mbox{$n=2$}.
Therefore, the sentence $\Phi(4)$ is true.
\begin{theorem}\label{the7}
The Conjecture is equivalent to the following conjecture on integer arithmetic:
if positive integers \mbox{$x_1,\ldots,x_n$} satisfy \mbox{${\rm max}(x_1,\ldots,x_n)>f(n)$},
then there exist positive integers \mbox{$y_1,\ldots,y_n$} such that
\[
\Bigl({\rm max}(x_1,\ldots,x_n)<{\rm max}(y_1,\ldots,y_n)\Bigl)~\wedge
\]
\[
\Bigl(\forall i,k \in \{1,\ldots,n\} ~(x_i+1=x_k \Longrightarrow y_i+1=y_k)\Bigr)~\wedge
\]
\[
\Bigl(\forall i,j,k \in \{1,\ldots,n\} ~(x_i \cdot x_j=x_k \Longrightarrow y_i \cdot y_j=y_k)\Bigr)
\]
\end{theorem}
\par
The execution of the following flowchart never terminates.
\newpage
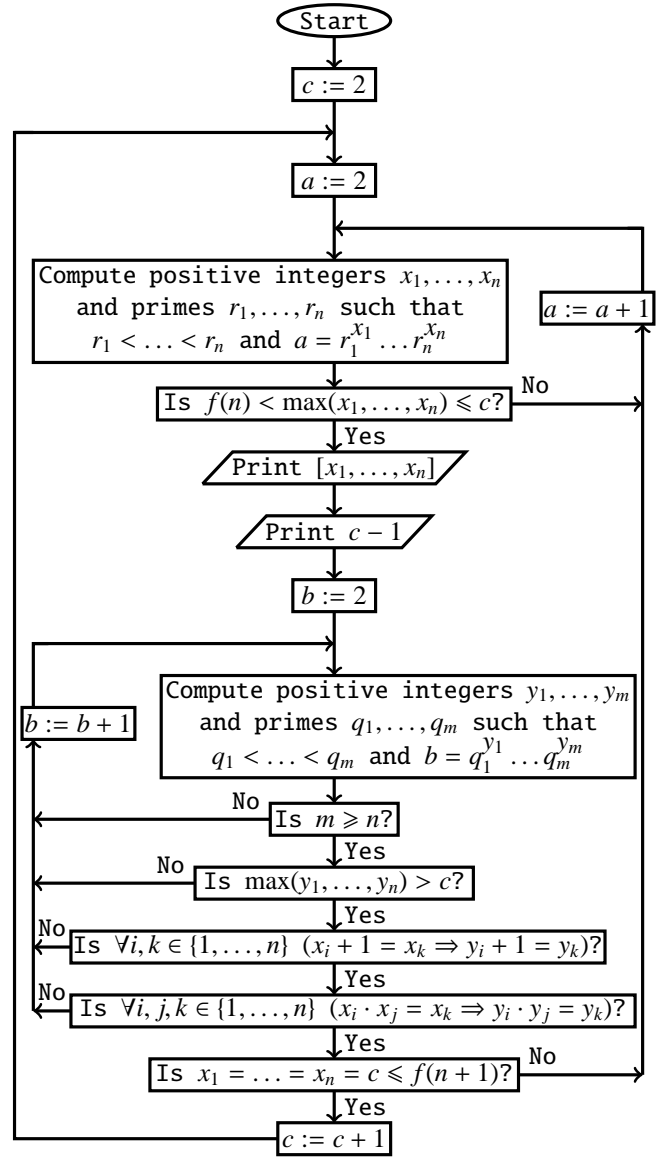
\begin{figure}[ht!]
\begin{tikzpicture}[very thick,yscale=.85]
\ttfamily
\path (0,17.75) node{Start};
\path (0,16.75) node{$c:=2$};
\path (0,15.25) node{$a:=2$};
\path (-.85,13.75) node{Compute positive integers $x_1,\ldots,x_n$};
\path (-.85,13.25) node{and primes $r_1,\ldots,r_n$ such that};
\path (-.85,12.75) node{$r_1<\ldots<r_n$ and $a=r_1^{\textstyle x_1} \ldots r_n^{\textstyle x_n}$};
\path (0,11.75) node{Is $f(n) < \max(x_1,\ldots,x_n) \leqslant c$?};
\path (0,10.75) node{Print $[x_1,\ldots,x_n]$};
\path (0,9.75) node{Print $c-1$};
\path (0,8.75) node{$b:=2$};
\path (.85,7.25) node{Compute positive integers $y_1,\ldots,y_m$};
\path (.85,6.75) node{and primes $q_1,\ldots,q_m$ such that};
\path (.85,6.25) node{$q_1<\ldots<q_m$ and $b=q_1^{\textstyle y_1} \ldots q_m^{\textstyle y_m}$};
\path (0,5.25) node{Is $m\geqslant n$?};
\path (0,4.25) node{Is $\max(y_1,\ldots,y_n)>c$?};
\path (.05,3.25) node{Is $\forall{i,k \in \{1,\ldots,n\}}\ (x_i+1=x_k \Rightarrow y_i+1=y_k)$?};
\path (.25,2.25) node{Is $\forall{i,j,k \in \{1,\ldots,n\}}\ (x_i \cdot x_j = x_k \Rightarrow y_i \cdot y_j = y_k)$?};
\path (0,1.25) node{Is $x_1 = \ldots = x_n = c \leqslant f(n+1)$?};
\path (0,.25) node{$c:=c+1$};
\path (3.5,13.25) node{$a:=a+1$};
\path (-3.4,6.75) node{$b:=b+1$};

\draw (0,17.75) ellipse (.75cm and 0.25cm);
\draw (-.55,16.5) rectangle (.55,17);
\draw (-.55,15) rectangle (.55,15.5);
\draw (-4,12.4) rectangle (2.3,14);
\draw (-2.35,11.5) rectangle (2.35,12);
\draw (-.55,8.5) rectangle (.55,9);
\draw (-2.3,5.9) rectangle (4,7.5);
\draw (-.85,5) rectangle (.85,5.5);
\draw (-1.85,4) rectangle (1.85,4.5);
\draw (-3.5,3) rectangle (3.6,3.5);
\draw (-3.5,2) rectangle (4,2.5);
\draw (-2.45,1) rectangle (2.45,1.5);
\draw (-.75,0) rectangle (.75,.5);
\draw (-4.15,6.5) rectangle (-2.65,7);
\draw (2.75,13) rectangle (4.25,13.5);
\draw (-1.75,10.5) -- (-1.35,11) -- (1.75,11) -- (1.35,10.5) -- (-1.75,10.5);
\draw (-1.3,9.5) -- (-.9,10) -- (1.3,10) -- (.9,9.5) -- (-1.3,9.5);
\draw (-.75,.25) -- (-4.25,.25) -- (-4.25,16) -- (-4,16);
\draw (-4,7) -- (-4,8) -- (-3.8,8)  (4.1,13.5) -- (4.1,14.5) -- (3.9,14.5);
\draw[->] (0,17.5) -- (0,17);
\draw[->] (0,16.5) -- (0,15.5);
\draw[->] (0,15) -- (0,14);
\draw[->] (0,12.4) -- (0,12);
\draw[->] (0,10.5) -- (0,10);
\draw[->] (0,9.5) -- (0,9);
\draw[->] (0,8.5) -- (0,7.5);
\draw[->] (0,5.9) -- (0,5.5);
\draw[->] (-4,2.25) -- (-4,6.5);
\draw[->] (4.1,1.25) -- (4.1,13);
\draw[->] (-4,8) -- (0,8);
\draw[->] (-4.25,16) -- (0,16);
\draw[->] (4.1,14.5) -- (0,14.5);
\draw[->] (0,11.5) -- (0,11) node[midway, right] {Yes};
\draw[->] (0,5) -- (0,4.5) node[midway, right] {Yes};
\draw[->] (0,4) -- (0,3.5) node[midway, right] {Yes};
\draw[->] (0,3) -- (0,2.5) node[midway, right] {Yes};
\draw[->] (0,2) -- (0,1.5) node[midway, right] {Yes};
\draw[->] (0,1) -- (0,.5) node[midway, right] {Yes};
\draw[->] (2.35,11.75) -- (4.1,11.75);
\draw[->] (-.86,5.25) -- (-4,5.25);
\draw[->] (-1.85,4.25) -- (-4,4.25);
\draw[->] (-3.5,3.25) -- (-4,3.25) node[midway, above] {No};
\draw[->] (-3.5,2.25) -- (-4,2.25) node[midway, above] {No};
\draw[->] (2.45,1.25) -- (4.1,1.25);
\path (2.35,11.75) node[above right] {No};
\path (-.85,5.25) node[above left] {No};
\path (-1.85,4.25) node[above left] {No};

\path (2.45,1.25) node[above right] {No}; 
\end{tikzpicture}
\end{figure}
\begin{theorem}\label{the8}
If the Conjecture is true, then the execution of the flowchart provides an infinite sequence
\mbox{$X_1,c_1,X_2,c_2,X_3,c_3,\ldots$} where \mbox{$\{c_1,c_2,c_3,\ldots\}=\N \setminus \{0\}$},
\mbox{$c_1 \leqslant c_2 \leqslant c_3 \leqslant \ldots$} and each~$X_i$ is a tuple of
positive integers. Each returned number~$c_i$ indicates that the performed computations
confirm the sentence $\Phi(c_i)$.
If the Conjecture is false, then the execution provides a similar finite sequence
$X_1,c_1,\ldots,X_k,c_k$ on the output. In this case,
for the tuple \mbox{$X_k=(x_1,\ldots,x_n)$} an appropriate tuple \mbox{$(y_1,\ldots,y_n)$} does not exist,
\mbox{$\{c_1, \ldots,c_k\}=[1,c_k] \cap \N$}, \mbox{$c_1 \leqslant \ldots \leqslant c_k$},
the sentences \mbox{$\Phi(1),\Phi(2),\Phi(3),\ldots,\Phi(c_k)$} are true, and the sentences
\mbox{$\Phi(c_k+1),\Phi(c_k+2),\Phi(c_k+3),\ldots$} are false.
\end{theorem}
\newpage
\begin{proof}
Let $p_n$ denote the \mbox{$n^{\rm th}$} prime number (\mbox{$p_1=2$}, \mbox{$p_2=3$}, etc.),
and let $c$ stands for any integer greater than $1$.
The function $f$ is strictly increasing and there exists the smallest positive integer $n$ such
that \mbox{$c \leqslant f(n+1)$}. Hence, if positive integers \mbox{$x_1,\ldots,x_i$} satisfy
\mbox{$f(i)<{\rm max}(x_1,\ldots,x_i) \leqslant c$}, then \mbox{$i \leqslant n$} and
\mbox{$2 \leqslant p_1^{\textstyle x_1} \ldots p_i^{\textstyle x_i} \leqslant p_1^{\textstyle c} \ldots p_n^{\textstyle c}$}.
Therefore, if the sentence \mbox{$\Phi(c)$} is true, then the flowchart algorithm checks
all tuples of positive integers needed to confirm the sentence \mbox{$\Phi(c)$}.
\end{proof}
\par
The following {\sl MuPAD} code implements a simplified flowchart's algorithm which checks
the following conjunction
\[
\Bigl(m \geqslant n\Bigr)~\wedge \Bigl({\rm max}(y_1,\ldots,y_n)>c\Bigr)~\wedge
\]
\[
\Bigl(\forall i,k \in \{1,\ldots,n\} ~(x_i+1=x_k \Longrightarrow y_i+1=y_k)\Bigr)~\wedge
\]
\[
\Bigl(\forall i,j,k \in \{1,\ldots,n\} ~(x_i \cdot x_j=x_k \Longrightarrow y_i \cdot y_j=y_k)\Bigr)
\]
instead of four separate conditions.
\begin{quote}
\begin{verbatim}
c:=2:
while TRUE do
a:=2:
repeat
S:=op(ifactor(a)):
n:=(nops(S)-1)/2:
u:=min(S[2*i+1] $i=1..n):
v:=max(S[2*i+1] $i=1..n):
X:=[S[2*i+1] $i=1..n]:
if n=1 then f:=1 end_if:
if n>1 then f:=2^(2^(n-2)) end_if:
if n>5 then f:=(2+2^(2^(n-4)))^(2^(n-4))
end_if:
g:=2^(2^(n-1)):
if n>4 then g:=(2+2^(2^(n-3)))^(2^(n-3))
end_if:
if f<v and v<=c then
print(X):
print(c-1):
b:=2:
repeat
T:=op(ifactor(b)):
m:=(nops(T)-1)/2:
Y:=[T[2*i+1] $i=1..m]:
r:=min(m-n+1,max(Y[i] $i=1..m)-c):
for i from 1 to min(n,m) do
for j from 1 to min(n,m) do
for k from 1 to min(n,m) do
if X[i]+1=X[k] and Y[i]+1<>Y[k] then
r:=0 end_if:








if X[i]*X[j]=X[k] and Y[i]*Y[j]<>Y[k] then
r:=0 end_if:
end_for:
end_for:
end_for:
b:=b+1:
until r>0 end_repeat:
end_if:
a:=a+1:
until c=u and c=v and c<=g end_repeat:
c:=c+1:
end_while:
\end{verbatim}
\end{quote}
\par
We attempt to confirm the sentence $\Phi(256)$.
Since the execution of the flowchart algorithm (or its any variant) proceeds slowly,
we must confirm the sentence $\Phi(256)$ in a different way.
For integers \mbox{$a_1,\ldots,a_n$}, let \mbox{$P(a_1,\ldots,a_n)$} denote
the following system of equations:
\[
\left\{\begin{array}{rcl}
x_{i}+1 &=& x_{k} {\rm ~~(if~} a_{i}+1=a_{k}) \\
x_{i} \cdot x_{j} &=& x_{k} {\rm ~~(if~} a_{i} \cdot a_{j}=a_{k})
\end{array}\right.
\]
\begin{lemma}\label{lem5}
For each positive integer $n$, there exist positive integers \mbox{$a_1,\ldots,a_n$} such that
\mbox{$a_1 \leqslant \ldots \leqslant a_n=\tau(n)$} and the system \mbox{$P(a_1,\ldots,a_n)$}
has only finitely many solutions in positive integers.
Each such numbers \mbox{$a_1,\ldots,a_n$} satisfy \mbox{$a_1<\ldots<a_n$}.
\end{lemma}
\begin{proof}
If \mbox{$a_1<\ldots<a_n$} does not hold, then we remove the first duplicate and insert \mbox{$a_n+1$} after $a_n$.
Since \mbox{$a_n+1>a_n=\tau(n)$}, we get a contradiction.
\end{proof}
\par
Let ${\cal F}$ denote the family of all systems \mbox{$P(a_1,a_2,a_3)$}, where integers
\mbox{$a_1,a_2,a_3$} satisfy \mbox{$1<a_1<a_2<a_3$}.
\begin{theorem}\label{the9}
The Conjecture is true for \mbox{$n=3$}.
\end{theorem}
\begin{proof}
By Lemma~\ref{lem5}, there exist positive integers $a_1$, $a_2$, $a_3$ such that \mbox{$a_1<a_2<a_3=\tau(3)$}
and the system \mbox{$P(a_1,a_2,a_3)$} has only finitely many solutions in positive integers.
If \mbox{$a_1=1$}, then \mbox{$a_2=2$} and \mbox{$a_3 \in \{3,4\}$}. Let \mbox{$a_1>1$}.
Since \mbox{$a_1<a_2<a_3$}, we get \mbox{$a_1 \cdot a_1<a_1 \cdot a_2<a_2 \cdot a_2$}. Hence,
\[
{\rm card}\Bigl(P(a_1,a_2,a_3) \cap \Bigl\{x_1 \cdot x_1=x_3,~x_1 \cdot x_2=x_3,~x_2 \cdot x_2=x_3\Bigr\}\Bigr) \leqslant 1
\]
Each integer $a_1$ satisfies \mbox{$a_1+1 \neq a_1 \cdot a_1$}. Hence,
\[
{\rm card}\Bigl(P(a_1,a_2,a_3) \cap \Bigl\{x_1+1=x_2,~x_1 \cdot x_1=x_2\Bigr\}\Bigr) \leqslant 1
\]
Since \mbox{$a_1<a_2<a_3$}, the equation \mbox{$x_1+1=x_3$} does not belong to \mbox{$P(a_1,a_2,a_3)$}.
By these observations, the following table shows all solutions in positive integers to any system that
belongs to ${\cal F}$.\\
\begin{table}[ht!]
\renewcommand{\arraystretch}{.95}
\setlength{\tabcolsep}{2.1pt}
\footnotesize
\centering
\begin{tabular}{|r|*{4}{c|}}
\cline{2-5}
\multicolumn{1}{c|}{$\phantom{\Bigg\vert}$} & $\varnothing$ & $\left\{x_1{\cdot}x_1{=}x_3\right\}$ & $\left\{x_1{\cdot}x_2{=}x_3\right\}$ & $\left\{x_2{\cdot}x_2{=}x_3\right\}$\\
 \hline
 & any triple & any triple & any triple & any triple \\
$\varnothing\, \cup$  & $\left(s,t,u\right)$ & $\left(s,t,s^2\right)$ & $\left(s,t,s{\cdot}t\right)$ & $\left(s,t,t^2\right)$\\
 & solves this	& solves this	& solves this	& solves this \\
 & system	& system	& system	& system \\
 \hline
 & any triple & any triple & any triple & any triple \\
$\left\{x_1{+}1{=}x_2\right\}\, \cup$  & $\left(s,s{+}1,u\right)$ & $\left(s,s{+}1,s^2\right)$ & $\left(s,s{+}1,s{\cdot}(s{+}1)\right)$ & $\left(s,s{+}1,(s{+}1)^2\right)$\\ 
 & solves this	& solves this	& solves this	& solves this \\
 & system	& system	& system	& system \\
 \hline
 & any triple &  	   & any triple & any triple \\
$\left\{x_1{\cdot}x_1{=}x_2\right\}\, \cup$  & $\left(s,s^2,u\right)$ & $\not\in {\cal F}$ & $\left(s,s^2,s^3\right)$ & $\left(s,s^2,s^4\right)$\\ 
 & solves this	&		& solves this	& solves this \\
 & system	&		& system	& system \\
 \hline
 & any triple & any triple &	& \\
$\left\{x_2{+}1{=}x_3\right\}\, \cup$  & $\left(s,t,t{+}1\right)$ & $\left(s,s^2{-}1,s^2\right)$ & $\not\in {\cal F}$ & $\not\in {\cal F}$\\ 
 & solves this	& solves this	&	& \\
 & system	& system	&	& \\
 \hline
 & any triple & only the triple &	& \\
$\{x_1{+}1{=}x_2,\phantom{\cup}$  & $\left(s,s{+}1,s{+}2\right)$ & $(2,3,4)$ & $\not\in {\cal F}$ & $\not\in {\cal F}$\\
$\phantom{\{}x_2{+}1{=}x_3\}\,\cup$ & solves this & solves this & & \\
 & system	& system	&	& \\
 \hline
 & any triple	&	&	& \\
$\{x_1{\cdot}x_1{=}x_2,\phantom{\cup}$ & $\left(s,s^2,s^2{+}1\right)$ & $\not\in {\cal F}$ & $\not\in {\cal F}$ & $\not\in {\cal F}$\\
$\phantom{\{}x_2{+}1{=}x_3\}\,\cup$ & solves this &	&	& \\
& system  	&	&	& \\
 \hline
\end{tabular}
\end{table}
\newpage
The table indicates that the system
\[
\Bigl\{x_1+1=x_2,~x_2+1=x_3,~x_1 \cdot x_1=x_3\Bigr\}=
\]
\[
\Bigl\{x_1+1=x_2,~x_2+1=x_3\Bigr\} \cup \Bigl\{x_1 \cdot x_1=x_3\Bigr\}
\]
has a unique solution in positive integers, namely \mbox{$(2,3,4)$}. The other presented systems
do not belong to ${\cal F}$ or have infinitely many solutions in positive integers.
\end{proof}
\begin{corollary}\label{cor3}
Since the Conjecture is true for \mbox{$n \in \{1,2,3\}$}, the sentence $\Phi(16)$ is true.
\end{corollary}
\begin{theorem}\label{the10}
The sentence $\Phi(256)$ is true.
\end{theorem}
\begin{proof}
By Corollary~\ref{cor3}, it suffices to consider quadruples of positive integers.
The next {\sl MuPAD} code returns $63$ quadruples \mbox{$(a_i,b_i,c_i,d_i)$} of positive integers,
where \mbox{$a_i<b_i<c_i<d_i \leqslant 256$} and \mbox{${\rm max}(a_i,b_i,c_i,d_i)=d_i>16$}. These quadruples
have the following property: if positive integers \mbox{$a,b,c,d$} satisfy \mbox{$a<b<c<d \leqslant 256$}
and \mbox{${\rm max}(a,b,c,d)=d>16$}, then there exists \mbox{$i \in \{1,\ldots,63\}$}
such that \mbox{$P(a,b,c,d) \subseteq P(a_i,b_i,c_i,d_i)$}.
\begin{quote}
\begin{verbatim}
TEXTWIDTH:=60:
S:={}:
G:=[]:
T:={}:
H:=[]:









for a from 1 to 256 do
for b from 1 to 256 do
for c from 1 to 256 do
Y:=[1,a+1,a*a,a*b]:
for l from 1 to 4 do
X:=sort([a,b,c,Y[l]]):
u:=nops({a,b,c,Y[l]}):
v:=max(a,b,c,Y[l]):
if u=4 and 16<v and v<257 then
M:={}:
for i from 1 to 4 do
for j from i to 4 do
for k from 1 to 4 do
if X[i]+1=X[k] then
M:=M union {[i,k]} end_if:
if X[i]*X[j]=X[k] then
M:=M union {[i,j,k]} end_if:
end_for:
end_for:
end_for:
d:=nops(S union {M})-nops(S):
if d=1 then
S:=S union {M}:
G:=append(G,M):
T:=T union {X}:
H:=append(H,X):
end_if:
end_if:
end_for:
end_for:
end_for:
end_for:
for w from 1 to nops(G) do
for z from 1 to nops(G) do
p:=nops(G[w] minus G[z]):
q:=nops(G[z] minus G[w]):
if p=0 and 0<q then T:=T minus {H[w]}
end_if:
end_for:
end_for:
print(T):
\end{verbatim}
\end{quote}
\noindent
The next table displays the quadruples $[a_1,b_1,c_1,d_1],\ldots,[a_{63},b_{63},c_{63},d_{63}]$
and shows that for each \mbox{$i \in \{1,\ldots,63\}$} the system
\mbox{$P(a_i,b_i,c_i,d_i)$} has infinitely many solutions in positive
integers, which completes the proof by Lemma~\ref{lem5}.
\begin{table}[ht!]
\renewcommand{\arraystretch}{.95}
\setlength{\tabcolsep}{1.7pt}
\centering
\footnotesize
\begin{tabular}{|*{4}{c|}}
\hline
$\left(1,\,2,\,3,\,t\right)$ & $\left(1,\,2,\,4,\,t\right)$ & $\left(2,\,3,\,4,\,t\right)$\\
$[1,\ 2,\ 3,\ 17]$ & $[1,\ 2,\ 4,\ 17]$ & $[2,\ 3,\ 4,\ 17]$\\\hline
$\left(t,\,t{+}1,\,t\left(t{+}1\right),\,t(t{+}1)^2\right)$ & $\left(1,\,2,\,t,\,2t\right)$ & $\left(1,\,t,\,t{+}1,\,t(t{+}1)\right)$\\
$[2,\ 3,\ 6,\ 18]$ & $[1,\ 2,\ 9,\ 18]$ & $[1,\ 4,\ 5,\ 20]$\\\hline
$\left(t,\,t^2,\,t^2{+}1,\,t^2\left(t^2{+}1\right)\right)$ & $\left(t,\,t{+}1,\,(t{+}1)^2,\,t(t{+}1)^2\right)$ & $\left(t,\,t{+}1,\,t{+}2,\,(t{+}1)(t{+}2)\right)$\\
$[2,\ 4,\ 5,\ 20]$ & $[2,\ 3,\ 9,\ 18]$ & $[3,\ 4,\ 5,\ 20]$\\\hline
$\left(1,\,2,\,t,\,t^2\right)$ & $\left(1,\,t,\,t{+}1,\,(t{+}1)^2\right)$ & $\left(t,\,t{+}1,\,t{+}2,\,t(t{+}1)\right)$\\
$[1,\ 2,\ 5,\ 25]$ & $[1,\ 4,\ 5,\ 25]$ & $[4,\ 5,\ 6,\ 20]$\\\hline
$\left(1,\,2,\,t,\,t{+}1\right)$ & $\left(t,\,t^2,\,t^2{+}1,\,\left(t^2{+}1\right)^2\right)$ & $\left(1,\,t,\,t{+}1,\,t^2\right)$\\
$[1,\ 2,\ 16,\ 17]$ & $[2,\ 4,\ 5,\ 25]$ & $[1,\ 5,\ 6,\ 25]$\\\hline
$\left(t,\,t{+}1,\,t{+}2,\,(t{+}2)^2\right)$ & $\left(1,\,t,\,t^2,\,t^2{+}1\right)$ & $\left(t,\,t^2,\,t^4,\,t^4{+}1\right)$\\
$[3,\ 4,\ 5,\ 25]$ & $[1,\ 4,\ 16,\ 17]$ & $[2,\ 4,\ 16,\ 17]$\\\hline
$\left(t,\,t{+}1,\,t{+}2,\,t(t{+}2)\right)$ & $\left(1,\,t,\,t^2,\,t^3\right)$ & $\left(t,\,t{+}1,\,(t{+}1)^2,\,(t{+}1)^2{+}1\right)$\\
$[4,\ 5,\ 6,\ 24]$ & $[1,\ 3,\ 9,\ 27]$ & $[3,\ 4,\ 16,\ 17]$\\\hline
$\left(t,\,t{+}1,\,t{+}2,\,(t{+}1)^2\right)$ & $\left(t,\,t{+}1,\,(t{+}1)^2,\,(t{+}1)^3\right)$ & $\left(t,\,t{+}1,\,t^2,\,t^2{+}1\right)$\\
$[4,\ 5,\ 6,\ 25]$ & $[2,\ 3,\ 9,\ 27]$ & $[4,\ 5,\ 16,\ 17]$\\\hline
$\left(t,\,t{+}1,\,t^2,\,t^3\right)$ & $\left(t,\,t{+}1,\,t{+}2,\,t^2\right)$ & $\left(t,\,t^2{-}1,\,t^2,\,t\left(t^2{-}1\right)\right)$\\
$[3,\ 4,\ 9,\ 27]$ & $[5,\ 6,\ 7,\ 25]$ & $[3,\ 8,\ 9,\ 24]$\\\hline
$\left(t,\,t{+}1,\,t^2,\,t(t{+}1)\right)$ & $\left(t,\,t^2,\,t^3,\,t^5\right)$ & $\left(t,\,t{+}1,\,t(t{+}1),\,t^2(t{+}1)^2\right)$\\
$[4,\ 5,\ 16,\ 20]$ & $[2,\ 4,\ 8,\ 32]$ & $[2,\ 3,\ 6,\ 36]$\\\hline
$\left(t,\,t^2{-}1,\,t^2,\,t^3\right)$ & $\left(t,\,t{+}1,\,t(t{+}1){-}1,\,t(t{+}1)\right)$ & $\left(1,\,t,\,t{+}1,\,t{+}2\right)$\\
$[3,\ 8,\ 9,\ 27]$ & $[4,\ 5,\ 19,\ 20]$ & $[1,\ 15,\ 16,\ 17]$\\\hline
$\left(t,\,t^2,\,t^2{+}1,\,t^3\right)$ & $\left(t,\,t{+}1,\,t^2,\,(t{+}1)^2\right)$ & $\left(t,\,t{+}1,\,t(t{+}1),\,t(t{+}1){+}1\right)$\\
$[3,\ 9,\ 10,\ 27]$ & $[4,\ 5,\ 16,\ 25]$ & $[4,\ 5,\ 20,\ 21]$\\\hline
$\left(t,\,t{+}1,\,t^2,\,(t{+}1)t^2\right)$ & $\left(t,\,t^2,\,t^2{+}1,\,t\left(t^2{+}1\right)\right)$ & $\left(t,\,t^2{-}1,\,t^2,\,t^2{+}1\right)$\\
$[3,\ 4,\ 9,\ 36]$ & $[3,\ 9,\ 10,\ 30]$ & $[4,\ 15,\ 16,\ 17]$\\\hline
$\left(t,\,t^2,\,t^4,\,t^5\right)$ & $\left(t,\,t{+}1,\,t(t{+}1),\,(t{+}1)^2\right)$ & $\left(1,\,t,\,t^2{-}1,\,t^2\right)$\\
$[2,\ 4,\ 16,\ 32]$ & $[4,\ 5,\ 20,\ 25]$ & $[1,\ 5,\ 24,\ 25]$\\\hline
$\left(t,\,t{+}1,\,t(t{+}1),\,t^2(t{+}1)\right)$ & $\left(t,\,t^2,\,t^2{+}1,\,t^2{+}2\right)$ & $\left(t,\,t{+}1,\,(t{+}1)^2{-}1,\,(t{+}1)^2\right)$\\
$[3,\ 4,\ 12,\ 36]$ & $[4,\ 16,\ 17,\ 18]$ & $[4,\ 5,\ 24,\ 25]$\\\hline
$\left(t,\,t{+}1,\,t^2{-}1,\,t^2\right)$ & $\left(t,\,t{+}1,\,t{+}2,\,t{+}3\right)$ & $\left(t,\,t^2,\,t^3{-}1,\,t^3\right)$\\
$[5,\ 6,\ 24,\ 25]$ & $[14,\ 15,\ 16,\ 17]$ & $[3,\ 9,\ 26,\ 27]$\\\hline
$\left(t,\,t^2,\,t^3,\,t^3{+}1\right)$ & $\left(t,\,t^2{-}2,\,t^2{-}1,\,t^2\right)$ & $\left(t,\,t^2,\,t^3,\,t^6\right)$\\
$[3,\ 9,\ 27,\ 28]$ & $[5,\ 23,\ 24,\ 25]$ & $[2,\ 4,\ 8,\ 64]$\\\hline
$\left(t,\,t^2{-}1,\,t^2,\,\left(t^2{-}1\right)^2\right)$ & $\left(t,\,t^2,\,t^4,\,t^6\right)$ & $\left(t,\,t^2{-}1,\,t^2,\,\left(t^2{-}1\right)t^2\right)$\\
$[3,\ 8,\ 9,\ 64]$ & $[2,\ 4,\ 16,\ 64]$ & $[3,\ 8,\ 9,\ 72]$\\\hline
$\left(t^2,\,t^3,\,t^4,\,t^6\right)$ & $\left(1,\,t,\,t^2,\,t^4\right)$ & $\left(t,\,t{+}1,\,(t{+}1)^2,\,(t{+}1)^4\right)$\\
$[4,\ 8,\ 16,\ 64]$ & $[1,\ 3,\ 9,\ 81]$ & $[2,\ 3,\ 9,\ 81]$\\\hline
$\left(t,\,t{+}1,\,t^2,\,t^4\right)$ & $\left(t,\,t^2{-}1,\,t^2,\,t^4\right)$ & $\left(t,\,t^2,\,t^2{+}1,\,t^4\right)$\\
$[3,\ 4,\ 9,\ 81]$ & $[3,\ 8,\ 9,\ 81]$ & $[3,\ 9,\ 10,\ 81]$\\\hline
$\left(t,\,t^2,\,t^3,\,t^4\right)$ & $\left(t,\,t^2,\,t^4{-}1,\,t^4\right)$ & $\left(t,\,t^2,\,t^4,\,t^8\right)$ \\
$[3,\ 9,\ 27,\ 81]$ & $[3,\ 9,\ 80,\ 81]$ & $[2,\ 4,\ 16,\ 256]$ \\\hline
\end{tabular}
\end{table}
\end{proof}
\newpage
Of course, the Conjecture restricted to integers \mbox{$n \in \{1,2,3,4\}$}
is intuitively obvious and implies Theorem~\ref{the10}. Formally,
the Conjecture remains unproven for \mbox{$n=4$}.
We explain why a hypothetical \mbox{brute-force} proof of the Conjecture for \mbox{$n=4$}
is much longer than the proof of Theorem~\ref{the10}. By Lemma~\ref{lem5},
it suffices to consider only the systems \mbox{$P(a,b,c,d)$}, where positive integers
\mbox{$a,b,c,d$} satisfy \mbox{$a<b<c<d$}.
\vskip 0.1truecm
\noindent
{\tt Case 1:} \mbox{$a=1$}. Obviously,
\vskip 0.1truecm
\centerline{${\rm card}\Bigl(\Bigl\{x_1+1=x_2\Bigr\} \cap P(a,b,c,d)\Bigr) \leqslant 1$}
\vskip 0.1truecm
\noindent
and
\[
{\rm card}\Bigl(\Bigl\{x_3+1=x_4\Bigr\} \cap P(a,b,c,d)\Bigr) \leqslant 1
\]
Since \mbox{$b+1<b \cdot b$}, we get
\[
{\rm card}\Bigl(\Bigl\{x_2+1=x_3,~x_2 \cdot x_2=x_3\Bigr\} \cap P(a,b,c,d)\Bigr) \leqslant 1
\]
Since \mbox{$b \cdot b<b \cdot c<c \cdot c$}, we get
\[
{\rm card}\Bigl(\Bigl\{x_2 \cdot x_2=x_4,~x_2 \cdot x_3=x_4,~x_3 \cdot x_3=x_4\Bigr\} \cap P(a,b,c,d)\Bigr) \leqslant 1
\]
The above inequalities allow one to determine
$(1+1) \cdot (1+1) \cdot (2+1) \cdot (3+1)=48$
systems which need to be solved.
\vskip 0.1truecm
\noindent
{\tt Case 2:} \mbox{$a>1$}. Obviously,
\vskip 0.1truecm
\centerline{${\rm card}\Bigl(\Bigl\{x_2+1=x_3\Bigr\} \cap P(a,b,c,d)\Bigr) \leqslant 1$}
\vskip 0.1truecm
\noindent
Since \mbox{$a+1<a \cdot a$}, we get
\[
{\rm card}\Bigl(\Bigl\{x_1+1=x_2,~x_1 \cdot x_1=x_2\Bigr\} \cap P(a,b,c,d)\Bigr) \leqslant 1
\]
Since \mbox{$c+1<a \cdot c$}, we get
\[
{\rm card}\Bigl(\Bigl\{x_3+1=x_4,~x_1 \cdot x_3=x_4\Bigr\} \cap P(a,b,c,d)\Bigr) \leqslant 1
\]
Since \mbox{$a \cdot a<a \cdot b<b \cdot b$}, we get
\[
{\rm card}\Bigl(\Bigl\{x_1 \cdot x_1=x_3,~x_1 \cdot x_2=x_3,~x_2 \cdot x_2=x_3\Bigr\} \cap P(a,b,c,d)\Bigr) \leqslant 1
\]
Since \mbox{$a \cdot a<a \cdot b<b \cdot b<b \cdot c<c \cdot c$}, we get
\[
{\rm card}\Bigl(\Bigl\{x_1 \cdot x_1=x_4,~x_1 \cdot x_2=x_4,~x_2 \cdot x_2=x_4,
\]
\[
x_2 \cdot x_3=x_4,~x_3 \cdot x_3=x_4\Bigr\} \cap P(a,b,c,d)\Bigr) \leqslant 1
\]
The above inequalities allow one to determine \mbox{$(1+1)~\cdot$}
\mbox{$(2+1) \cdot (2+1) \cdot (3+1) \cdot (5+1)=432$} systems which need to be solved.
\vskip 0.1truecm
\par
{\sl MuPAD} is a computer algebra system whose syntax is modelled on {\sl Pascal}.
The commercial version of {\sl MuPAD} is no longer available as a \mbox{stand-alone}
product, but only as the {\sl Symbolic Math Toolbox} of {\sl MATLAB}. Fortunately,
the presented codes can be executed by {\sl MuPAD Light},
which was offered for free for research and education until autumn 2005.


\begin{thebibliography}{11}

\bibitem{DMR}
\mbox{M. Davis}, \mbox{Yu. Matiyasevich}, \mbox{J. Robinson},
\newblock
{\em Hilbert's tenth problem. Diophantine equations: positive aspects of a negative solution,}
in: Mathematical developments arising from Hilbert problems (ed.~F.~E.~Browder),
\newblock
Proc. Sympos. Pure Math., vol. 28, Part 2, Amer. Math. Soc., 1976, \mbox{323--378};
\newblock
reprinted in: The collected works of Julia Robinson (ed.~S.~Feferman), Amer. Math. Soc., 1996, \mbox{269--324}.

\bibitem{Matiyasevich1}
\mbox{Yu. Matiyasevich},
\newblock
{\em Hilbert's tenth problem,}
\newblock
MIT Press, Cambridge, MA, 1993.

\bibitem{Matiyasevich2}
\mbox{Yu. Matiyasevich},
\newblock
{\em Hilbert's tenth problem: what was done and what is to be done.}
\newblock
Hilbert's tenth problem: relations with arithmetic and algebraic geometry (Ghent, 1999), \mbox{1--47},
\newblock
Contemp. Math. 270, Amer. Math. Soc., Providence, RI, 2000.

\bibitem{Matiyasevich3}
\mbox{Yu. Matiyasevich},
\newblock
{\em Towards \mbox{finite-fold} Diophantine representations,}
\newblock
J.~Math. Sci. (N. Y.) \mbox{vol. 171}, \mbox{no. 6}, 2010, \mbox{745--752},
\newblock
\url{http://dx.doi.org/10.1007%2Fs10958-010-0179-4}.

\bibitem{Robinson}
\mbox{J. Robinson},
\newblock
{\em Definability and decision problems in arithmetic,}
\newblock
J.~Symbolic Logic 14 (1949), \mbox{98--114};
\newblock
reprinted in: The collected works of Julia Robinson (ed.~S.~Feferman),
Amer. Math. Soc., 1996, \mbox{7--23}.

\bibitem{Skolem}
\mbox{Th. Skolem},
\newblock
{\em Diophantische Gleichungen},
\newblock
Julius Springer, Berlin, 1938.

\bibitem{Smart}
\mbox{N. P. Smart},
\newblock {\em The algorithmic resolution of Diophantine equations,}
\newblock Cambridge University Press, Cambridge, 1998,
\newblock
\url{http://dx.doi.org/10.1017/CBO9781107359994}.

\bibitem{Stoll}
\mbox{M. Stoll},
\newblock {\em How to solve a Diophantine equation,}
in: An invitation to mathematics: From competitions to research (eds. M.~Lackmann and D.~Schleicher),
\newblock
Springer, Berlin-Heidelberg-New York, 2011, \mbox{9--19},
\newblock
\url{http://dx.doi.org/10.1007/978-3-642-19533-4_2}.

\bibitem{Tyszka1}
\mbox{A. Tyszka},
\newblock {\em All functions \mbox{$g \colon \N \to \N$} which have a \mbox{single-fold} Diophantine
representation are dominated by a \mbox{limit-computable} function \mbox{$f \colon \N \setminus \{0\} \to \N$}
which is implemented in {\sl MuPAD} and whose computability is an open problem,}
in: Computation, cryptography, and network security (eds. N.~Daras and M.~Th.~Rassias),
Springer, Berlin-Heidelberg-New York, 2015, in print.

\bibitem{Tyszka2}
\mbox{A. Tyszka},
\newblock
{\em Conjecturally computable functions which unconditionally
do not have any \mbox{finite-fold} Diophantine representation,}
\newblock
Inform. Process. Lett. 113 (2013), \mbox{no. 19--21}, \mbox{719--722},
\newblock
\url{http://dx.doi.org/10.1016/j.ipl.2013.07.004}.

\balance

\bibitem{Tyszka3}
\mbox{A. Tyszka},
\newblock
{\em Does there exist an algorithm which to each Diophantine equation
assigns an integer which is greater than the modulus of integer solutions,
if these solutions form a finite set?}
\newblock
Fund. Inform. 125 (2013), \mbox{no. 1}, \mbox{95--99},
\newblock
\url{http://dx.doi.org/10.3233/FI-2013-854}.

\end{thebibliography}
\end{document}